\documentclass[a4paper,11pt]{article}

\usepackage{amsthm}
\usepackage{amsmath}
\usepackage{latexsym}
\usepackage{amssymb}
\usepackage[ansinew]{inputenc}

\newcommand{\nui}[1]{N\paren{#1}}

\newtheorem{fed}{Definition}[section]
\newtheorem*{fed*}{Definition}
\newtheorem*{feds*}{Definitions}
\newtheorem{teo}[fed]{Theorem}
\newtheorem*{teo*}{Theorem}
\newtheorem{lem}[fed]{Lemma}

\theoremstyle{definition}
\newtheorem{rem}[fed]{Remark}

\newtheorem*{rems*}{Remarks}

\def\bdem{\begin{proof}}
	\def\edem{\end{proof}}

\def\coma{\, , \, }

\def\py{\peso{and}}
\newcommand{\peso}[1]{ \quad \text{ #1 } \quad }

\def\suml{\sum\limits}

\def\bce{\begin{center}}
\def\ece{\end{center}}

\def\cO{{\mathcal O}}

\def\py{\peso{and}}

\def\noi{\noindent}
\def\cF{\mathcal F}
\def\cG{\mathcal G}

\def\QED{\hfill $\square$}
\def\EOE{\hfill $\triangle$}

\def\bm{\left[\begin{array}}
\def\em{\end{array}\right]}
\def\ben{\begin{enumerate}}
\def\een{\end{enumerate}}
\def\bit{\begin{itemize}}
\def\eit{\end{itemize}}
\def\barr{\begin{array}}
\def\earr{\end{array}}

\def\eps{\varepsilon}
\def\la{\lambda}

\def\R{\mathbb{R}}

\def\C{\mathbb{C}}
\def\I{\mathbb{I}}

\def\cP{\mathcal{P}}

\def\cS{{\cal S}}
\def\cT{{\cal T}}
\def\cM{{\cal M}}

\def\cN{{\cal N}}
\def\cV{{\cal V}}
\def\cU{{\cal U}}

\def\ua{^\uparrow}
\def\da{^\downarrow}

 \DeclareMathOperator{\Pim}{Im}
 \DeclareMathOperator{\tr}{tr}

\DeclareMathOperator{\leqp}{\leqslant}

\newcommand{\mat}{\mathcal{M}_d(\mathbb{C})}

\newcommand{\matsad}{\mathcal{H}(d)}

\newcommand{\matud}{\mathcal{U}(d)}

\newcommand{\matpos}{\mat^+}
\newcommand{\defpos}{\cP(d)}

\def\beq{\begin{equation}}
\def\eeq{\end{equation}}

\def\pausa{\medskip\noi}

\newcommand\restr[2]{{  \left.\kern-\nulldelimiterspace #1 
  \vphantom{\big|}   \right|_{#2} }}

\def\Ax2{\,( S_{E(\cF)^\#_\cV})\hat{}_x }

\newcommand{\corch}[1]{\left[ #1 \right]}
\newcommand{\paren}[1]{\left(#1\right)}
\newcommand{\llav}[1]{\left\{#1\right\}}

\newcommand{\norm}[1]{\left\|#1\right\|}

\def\nug0{\nu(\cG_0)}

\begin{document}

\title{ {\bf Local extrema for Procustes problems in the set of positive definite matrices}}
\author{ Pablo Calder\'on $^{*}$, Noelia B. Rios $^{*}$ and Mariano A. Ruiz
	\footnote{ e-mail addresses: pcalderon@mate.unlp.edu.ar, nbrios@mate.unlp.edu.ar, mruiz@mate.unlp.edu.ar}
	\\ {\small Depto. de Matem\'atica, FCE-UNLP
		and IAM-CONICET, Argentina  }}
	\date{}
\maketitle

\begin{abstract}
Given two positive definite matrices $A$ and $B$, a well known result by Gelfand, Naimark and Lidskii establishes a relationship between the eigenvalues of $A$ and $B$ and those of $AB$ by means of majorization inequalities. In this work we make a local study focused in the spectrum of the matrices that achieve the equality in those inequalities. As an application, we complete some previous results concerning Procustes problems for unitarily invariant norms in the manifold of positive definite matrices.
\end{abstract}


\section{Introduction}

Majorization (and log-majorization) between vectors in $\R^d$ plays a fundamental role in matrix analysis. Although it is not a total order, it is well known that several matrix inequalities are consequences of the comparison of the singular values by means of majorization.
Lidskii's  inequalities  are certainly a very good example of this fact. Indeed, they are essential in some  natural  norm inequalities in matrix theory: those derived from  proximity problems (matrix nearness problems) such as, for example, Procustes problems (\cite{Bhat,GB,Ka}).
In this case, Lidskii inequality provides  an explicit description of global minimizers of functions defined on unitary orbits of matrices. These functions are constructed from the distance to minimize, using  (strictly convex) unitarily invariant norms  (see for example \cite{mrs2}). When the unitary orbit is endowed with a metric space structure, a local study on Lidskii's Theorems also allows a characterization of the spectral structure of local minimizers of such functions. 

Given $A\in\mat$, a complex matrix of size $d$, 
a norm $N(\cdot)$ in $\mat$ 
and a set $\mathcal X\subset \mat$, the typical matrix approximation problem is to search for the minimal distance 
$$ 
d_{N}(A,\mathcal X)=\min \{ d_N(A,C):=N(A-C)\,:\ C\in\mathcal X\} \ ,
$$ 
and for the best approximants
of $A$ in $\mathcal X$: 
$$ 
\mathcal A^{\rm op}_{N}(A\, , \, \mathcal X) 
=\{C\in\mathcal X:\ d_N(A,C)= d_{N}(A \, , \, \mathcal X)\}\ .
$$
That is, solving these problems requires to
provide a characterization and, if possible, an explicit computation (in some cases sharp estimations) 
of $d_{N}(A \coma \mathcal X)$ and  the set of best approximations 
$\mathcal A^{\rm op}_{N}(A\coma \mathcal X)$. 
Typically, a  choice for $N$ is the Frobenius norm (also called $2$-norm) since it is the norm associated with an inner product 
in $\mat$. Other norms that are also of interest are  weighted norms, the $p$-norms for $1\leq p$ (which contains the Frobenius norm), or the 
more general class of unitarily invariant norms. With regards to the sets $\mathcal X$, the most relevant choices are: selfadjoint matrices, 
positive semidefinite matrices, correlation matrices, orthogonal projections, oblique projections and matrices with rank bounded by a fix number 
(see for example \cite{CMR,Eld,High,Ki,Wat}).  In the special case of Procustes type problem, $\mathcal X$ is the unitary orbit of a positive (or selfadjoint) matrix $B$ and $A$ is also a positive (or selfadjoint) matrix.

Once the nearness problem above has been solved for some $A$,  $\mathcal X$ and  $N(\cdot)$, it is natural to pose the following proximity problem: given  a (fixed) $C_0\in\mathcal X$, the problem is to compute (or at least to give an upper bound of) the distance 
$$ 
d_\mathcal X(C_0\coma  \mathcal A^{\rm op}_{N}(A\coma \mathcal X))
=\min\{d_\mathcal X(C_0\coma C):\ C\in \mathcal A^{\rm op}_{N}(A\coma \mathcal X)\}\ ,
$$
where $d_\mathcal X$ denotes a metric in the set $\mathcal X$.

In \cite{dnp1} the following matrix nearness problem is  considered:
let $\nui{\cdot}$ be an arbitrary  unitarily invariant norm in $\mat$. Consider two arbitrary positive semidefinite matrices $A,B \in\matpos$ and let 
$$\cO_B:=\llav{C\in\matpos:\, C=U^*BU \quad\text{for}\,\, U\,\text{ unitary}}\,.$$
Note that $\cO_B$ is a metric space endowed with the usual  metric induced by the operator norm.
Defining the distance 
$d_{\paren{N,\, A,\, B}}=d_N:\cO_B\to \R_{>0}$ given by
$$
d_N(C)= N\paren{A-C}\peso{for} C\in\cO_B\,,
$$ 
thus, the  Procustes problem in $\matpos$ is given by
$$
\min_{C\in \cO_B} d_N\paren{C}
=\min\llav{N\paren{A-C}:\, C\in \cO_B}\,.
$$

With this notation, the nearness problem for $\mathcal X=\cO_B\subset \matpos$,  is solved for an 
arbitrary strictly convex unitarily invariant norm $N(\cdot)$ in $\mat$. That means that  an explicit description
of $d_{N}(A\coma \cO_B)$ and 
$\mathcal A^{\rm op}_{N}(A\coma \cO_B)$ is obtained. In that work,  local minimizers of $d_N$ in $\cO_B$ are completely characterized in terms of their spectra as an application of local Lidskii's (additive) Theorems. Even more, local minimizers are in effect global and it does not depend on  the (strictly convex) unitary invariant norm chosen.

In this paper, based on \cite{BhatCon} and using similar techniques as in \cite{dnp1, dnp2, dnp3}, we focus on the following Procustes problem: let $N(\cdot)$ a unitary invariant norm and $\defpos$ the set of the strictly positive matrices of dimension $d$. 
Then, given  $A,B\in\defpos$,  define the distance 
$F_{\paren{N,\, A,\, B}}=F_N:\cO_B\to \R$ given by
\beq
F_N(C)= N\paren{\log\paren{A^{-1/2}CA^{-1/2}}}\peso{for} C\in\cO_B\,.
\eeq
Then, the goal is to compute
\beq
\min_{C\in \cO_B} F_N\paren{C}=\min_{C\in \cO_B}N\paren{\log\paren{A^{-1/2}CA^{-1/2}}}\,.
\eeq
and to characterize the minimizers when $N(\cdot)$ is strictly convex.

Our purpose in this note is to make a local analysis of the Gelfand-Naimark-Lidskii result, using a similar approach  to that applied to the study of the (additive) Lidskii inequalities in \cite{dnp1}. Then we use this  in order to extend the results in \cite{BhatCon} to a characterization of the local extrema for the Procustes problem described above.

\section{Preliminaries}

In this section we introduce the notations, terminology and results 
that we will use throughout the 
paper.

\pausa
{\bf Notation and terminology}. We let $\mathcal M_{k,d}(\C)$ be the space of complex $k\times d$ matrices and write $\mathcal M_{d,d}(\C)=\mat$ for the algebra of complex $d\times d$  matrices. We denote by $\matsad\subset \mat$ the real subspace of selfadjoint matrices and by $\defpos$ the set of positive definite matrices in $\mat$. We let $\matud\subset \mat$ denote the group of unitary matrices.

Given $x=(x_i)_{i\in\I_d}\in\R^d$ we denote by $x\da=(x_i\da)_{i\in\I_d}$ the vector obtained by rearranging the entries of $x$ in non-increasing order (analogously $x\ua$ denote in non-decreasing).  

Given a matrix $A\in\matsad$ we denote by 
$\la(A)=\la(A)\da=(\la_i(A))_{i\in\I_d}\in (\R^d)\da$ the eigenvalues of $A$ counting multiplicities and arranged in 
non-increasing order.  
Given a vector $x\in\C^d$ we denote by $D_x$ the diagonal matrix in $\mat$ whose main diagonal is given by $x$.

We denote by $(\R^d)\da=\{x\da:\ x\in\R^d\}$ and $(\R_{\geq 0}^d)\da=\left\{x\da:\ x\in\R_{\geq 0}^d\right\}$.
If $x,\,y\in\C^d$ we denote by $x\otimes y\in\mat$ the rank-one matrix given by $(x\otimes y) \, z= \langle z\coma y\rangle \ x$, for $z\in\C^d$.

\pausa

Next we recall the notion of majorization between vectors, which is one of the main characters throughout our work.
\begin{fed}\rm 
	Let $x,y\in\R^d$. We say that $x$ is
	{\it submajorized} by $y$, and write $x\prec_w y$,  if
	$$
	\suml_{i=1}^j x^\downarrow _i\leq \suml_{i=1}^j y^\downarrow _i \peso{for every} 1\leq j\leq d\,.
	$$
	If $x\prec_w y$ and $\tr x = \sum_{i=1}^d x_i=\sum_{i=1}^d y_i = \tr y$,  then $x$ is
	{\it majorized} by $y$, and write $x\prec y$.
\end{fed}

\begin{rem}\label{desimayo}
	\pausa Given $x,\,y\in\R^d$ we write
	$x \leqp y$ if $x_i \le y_i$ for every $i\in \mathbb I_d \,$.  It is a standard  exercise 
	to show that: 
	\begin{enumerate}
		\item $x\leqp y \implies x^\downarrow\leqp y^\downarrow  \implies x\prec_w y $. 
		\item $x\prec y\implies |x|\prec_w|y|$, where $|x|=(|x_i|)_{i\in\I_d}\in\R_{\geq 0}^d$.
		\item $x\prec y,\, |x|\da=|y|\da \implies x\da=y\da$.
		\EOE
	\end{enumerate}
\end{rem}

\pausa

Recall that a norm $\nui{\cdot}$ in $\mat$ is unitarily invariant if 
$$ \nui{UAV}=\nui{A} \peso{for every} A\in\mat \py U,\,V\in\matud\,,$$
and $\nui{\cdot}$ is {\bf strictly convex} if its restriction to diagonal matrices is a strictly convex  norm in $\C^d$. 
Examples of u.i.n. are the spectral norm $\|\cdot\|$ and the $p$-norms $\|\cdot\|_p$, for $p\geq 1$
(strictly convex if $p>1$).
It is well known that (sub)majorization relations between singular values of matrices are intimately related 
with inequalities with respect to u.i.n's. 
The following result summarizes these relations (see for example \cite{Bhat}):

\begin{teo}\label{teo intro prelims mayo}\rm
	Let $A,\,B\in\mat$ be such that $s(A)\prec_w s(B)$ where
	$s(C)=\la(|C|)$ denotes the singular values of $C$, i.e. the eigenvalues of $|C|=(C^*C)^{1/2}\in\matpos$. Then:
	\ben 
	\item For every u.i.n. $\nui{\cdot}$ in $\mat$
	we have that $N(A)\leq N(B)$.
	\item If $\nui{\cdot}$ is a strictly convex u.i.n. in $\mat$ 
	and $N(A)=N(B)$, then  $s(A)=s(B)$.\qed
	\een
	
\end{teo}

Given a differential manifold $\cN$ and $p\in\cN$, we denote the tangent space of $\cN$ in $p$ by $\cT_p \,\cN$.
If $\cN, \cM$ are differential manifolds and $F:\cN\to\cM$ is a differential map; recall that  $F$ is a  submersion at $p\in\cN$ if the differential 
$D_{p}F:\cT_{p}\,\cN\to \cT_{F(p)}\,\cM,$ 
is surjective. If $F$ is a submersion at every $p\in\cN$ then $F$ is a \textit{submersion}.

As a consequence,
\begin{teo}\label{es abierta}\rm
	Let $\cN,\cM$ differential manifolds and $F:\cN\to\cM$ a submersion at $p_0\in\cN$. Then,
	$F$ is locally open around $p_0$; i.e, $\forall \eps>0$,  given 
	$
	\cN_\eps:=\llav{p\in\cM:\, d(p,p_0)<\eps},
	$
	the set $F(\cN_\eps)$
	contains an open neighborhood of $F(p_0)$ in $\cM$.
	\QED
\end{teo}
\pausa

In this paper we consider  $\matud\in\mat$ endowed with its natural (Lie) structure of differential manifold.
Thus, we consider the product manifold
$\matud\times\matud$ endowed with the following metric
$$d((U_1,V_1),(U_2,V_2))=\max\{\|U_1-U_2\|,\,\|V_1-V_2\| \}\,
\peso{for} (U_1,V_1),(U_2,V_2) \in \matud\times\matud\,$$ and $\|\cdot\|$ the usual spectral norm.
It is known that the exponential map
$\matsad\ni i\cdot X\mapsto \exp(X)$ identifies the tangent space
$\cT_I \,\matud$ ($I$ the identity)
with the set of anti-hermitian matrices  $i\cdot\matsad$,
since the curve $\gamma(t)=\exp(tX)\in\matud$ is such that $\gamma'(0)=X\in i\cdot\matsad$.

Given $A,B\in\matsad$, let $\cO_A=\llav{V^*AV:\, V\in\matud }$ be the unitary orbit of $A$, and  consider the smooth function $\mathcal C_A:\matud\to \cO_A$, defined by $\mathcal C_A(U)=U^*AU$.

Then
$$
D_I \mathcal C_A (X)=\corch{X,A}\in\matsad\,\peso{for} X\in i\cdot\matsad\,,
$$
where $\corch{B,A}:=BA-AB$.

Given a set of matrices $\mathcal S=\cS^*=\llav{A\in \mat:\, A^*\in\cS}\subset\mat$, the commutant of $\mathcal S$ is the
unital $*$-subalgebra of $\mat$ given by 
$$ \mathcal S'=\{\ C\in\mat:\ [C,D]=0\ \text{for all}\ D\in\mathcal S\ \}\subset \mat\,.$$
\section{A geometrical study of Gelfand - Naimark - Lidskii inequality}
Given $A,B$ in $\defpos$, the well known  Gelfand-Naimark-Lidskii result relates the eigenvalues of  $A, B$  and $AB$ by means of log-majorization inequalities. 

\begin{teo}[Gelfand - Naimark - Lidskii]
	Let $A,B$ be in $\defpos$.  Then
	\[ \log (\la^\downarrow (B))+\log(\la^\uparrow (A))\prec \log (\la(BA))
	\prec \log (\la^\downarrow (B))+\log(\la^\downarrow (A))
	.\]
\end{teo}
Notice that, in particular, 
\begin{equation}\label{Lidskii multiplicativo} \log (\la^\downarrow (B))-\log(\la^\downarrow (A))\prec \log (\la(BA^{-1}))
\prec \log (\la^\downarrow (B))-\ln(\la^\uparrow (A)).
\end{equation}
The purpose of this section is to do a {\bf local} study of the equality case: 
$$\log (\la^\downarrow (B))-\log(\la^\downarrow (A))= \log (\la(BA^{-1})).$$ We shall restrict our attention to this inequality, needless to say that a similar study can be done for the equality in  the right majorization inequality
in \eqref{Lidskii multiplicativo}.

The next result of \cite{MRS} characterizes the equality. For convenience, we restate it in our log notation for GNL inequality. Recall that in our notation we consider the eigenvalues arranged in a non-increasing order:  $\la(A)=\la^\downarrow (A)$ and $\la(B)=\la^\downarrow (B)$.
\begin{teo} [\cite{MRS}, Theorem 5.1]\label{Teo MRS}
	Let $A$, $B$ $\in\defpos$. Then,
	\[\left( \log (\la(B)) - \log(\la(A))\right)^\downarrow =\log (\la (BA^{-1}))\] 
	if and only if there exists an unitary $U\in \matud$ such that 
	\[UAU^*=D_{\la(A)}\text{ and } UBU^*=D_{\la(B)}.\] 
\end{teo}

Following the techniques developed in \cite{dnp2}, we consider the manifold
$$\defpos ^\tau:=\{C\in \defpos \,:\, \det(C)=\tau\}$$ with $\tau= 
\frac{\det(B)}{\det(A)}$. Then we define a map
$\Gamma : \matud \times \matud \longrightarrow \defpos^\tau $ given by
\begin{equation}\label{el gama multi}
\Gamma(U,V)=UA^{-1/2}U^*VBV^*UA^{-1/2}U^*.
\end{equation}
\begin{lem} \label{gama es una submersion}
	$\Gamma$ is a submersion at $(I,I)\in \matud\times \matud$ if and only if 
	\[ \{A,B\}':=\{C\in \mat \, :\, AC=CA, \, BC=CB\}=\C\cdot I.\]
\end{lem}
\begin{proof}
	Suppose that the relative commutant $\{A,B\}'$ is trivial. In order to prove that $\Gamma$ is a submersion at $(I,I)$, we need to show that the differential map 
	\[D_{(I,I)}\Gamma : \cT_{(I,I)}\, \matud \times \matud  \longrightarrow \cT_{A^{-1/2}BA^{-1/2}}\, \defpos^\tau\]
	is surjective.
	
	Suppose that $D_{(I,I)}\Gamma$ is not surjective. Using the Riemannian structure of the tangent space at $\defpos^\tau$, there exists $Y_0\in \cT_{A^{-1/2}BA^{-1/2}}\, \defpos^\tau$ such that 
	\[\langle D_{(I,I)}\Gamma (X_1,X_2)\coma Y_0\rangle =\tr\left(D_{(I,I)}\Gamma (X_1,X_2) \, Y_0^*\right)=0,\]
	for every $(X_1,X_2)\in \cT_{(I,I)}\, \matud\times \matud$. In particular, 
	\[\langle D_{(I,I)}\Gamma (X,0)\coma Y_0\rangle=\langle D_{(I,I)}\Gamma (0,X)\coma Y_0\rangle =0,\]
	for every $X\in \mat$, $X^*=-X$.
	Since 
	\[\cT_{(I,I)}\, \matud\times\matud =\{(X_1,X_2)\in \mat\times \mat\, :\, X_i^*=-X_i, i=1,2\}\]
	and by standard computation one has that
	\[\cT_{A^{-1/2}BA^{-1/2}}\, \defpos^\tau=
	\left\{ Y\in \matsad : \, \tr \left((A^{1/2}B^{-1}A^{1/2})\,Y\right)=0\right\},\]
	$Y_0\in \matsad$ is such that $\tr \left(C\,Y_0\, C\right)=0$, with $C=(A^{-1/2}BA^{-1/2})^{-1/2}$. Therefore, since 
	\begin{equation}\label{derivada 1}
	D_{(I,I)}\Gamma (0,X)=A^{-1/2}[B,X]A^{-1/2},
	\end{equation}
	and
	\begin{equation}\label{derivada 2}
	D_{(I,I)}\Gamma (X,0)=[A^{-1/2},X]BA^{-1/2}+A^{-1/2}B[A^{-1/2},X],
	\end{equation}
	we have that
	\[\tr(A^{-1/2}[B,X]A^{-1/2}Y_0)=\tr([A^{-1/2}Y_0A^{-1/2},B]X)=0\]
	for all $X^*=-X$. In particular, for $X=[A^{-1/2}Y_0A^{-1/2},B]$, we deduce
	\begin{equation} \label{conmuta 1}
	[A^{-1/2}Y_0A^{-1/2},B]=0.
	\end{equation}
	On the other side, 
	\begin{align*}
	\tr([A^{-1/2},X]BA^{-1/2}Y_0)+\tr(A^{-1/2}&B[A^{-1/2},X]Y_0)=\\
	=\tr(([BA^{-1/2}&Y_0,A^{-1/2}]+[Y_0A^{-1/2}B,A^{-1/2}])X)=0
	\end{align*}
	for every $X^*=-X$. Thus, we have
	\begin{equation}\label{conmuta 2}
	[BA^{-1/2}Y_0,A^{-1/2}]+[Y_0A^{-1/2}B,A^{-1/2}]=0.
	\end{equation}
	From \eqref{conmuta 1} and \eqref{conmuta 2} we deduce
	$[A^{-1/2}BA^{-1/2},Y_0]=0,$
	that is,
	\[BA^{-1/2}Y_0A^{-1/2}A=AA^{-1/2}Y_0A^{-1/2}B.\]
	Finally, since $[A^{-1/2}Y_0A^{-1/2},B^{1/2}]=0$, we deduce that
	\[B^{1/2}A^{-1/2}Y_0A^{-1/2}B^{1/2}\in \{A,B\}'.\]
	Notice that, since 
	\[0=\tr(A^{-1/2}BA^{-1/2}Y_0)=\tr(B^{1/2}A^{-1/2}Y_0A^{-1/2}B^{1/2}),\]
	we have that $B^{1/2}A^{-1/2}Y_0A^{-1/2}B^{1/2}$ is not a scalar. Hence, we arrive to a contradiction and $\Gamma$ is a submersion at $(I,I)$.
	
	Conversely, suppose that there exists a non-trivial orthogonal projection $P$ in $\{A,B\}'$. Let $Z\in \matsad\cap \{A,B\}'$ be defined as
	$Z:=P+\frac{c}{c- \tr((A^{1/2}B^{-1}A^{1/2})^2)} (I-P)$, where $c=\tr(P(A^{1/2}B^{-1}A^{1/2})^2)$ (notice that, by functional calculus, $P$ commutes with $A^{1/2}B^{-1}A^{1/2}$).
	
	Then, it holds that $Y_0:=A^{1/2}B^{-1/2}ZB^{-1/2}A^{1/2}=A^{1/2}B^{-1}A^{1/2}Z$ is such that $Y_0 \in \matsad$. It is clear that $Y_0\neq 0$. Moreover,
	\begin{align*}
	\tr(A^{1/2}B^{-1}A^{1/2} Y_0)&=\tr\left((A^{1/2}B^{-1}A^{1/2})^2Z\right)\\
	&=\tr\left(P(A^{1/2}B^{-1}A^{1/2})^2\right)+c \tr\left((I-P)(A^{1/2}B^{-1}A^{1/2})^2\right)\\                    
	&=c+c\left(\frac{\tr((A^{1/2}B^{-1}A^{1/2})^2)}{c-\tr((A^{1/2}B^{-1}A^{1/2})^2)}- \frac{c}{c-\tr((A^{1/2}B^{-1}A^{1/2})^2)}\right)=0.
	\end{align*}
	That is, $Y_0\in \cT_{A^{-1/2}BA^{-1/2}}\defpos^\tau$. It is easy to see that
	\[ [A^{-1/2}Y_0A^{-1/2}, B]=[BA^{-1/2}Y_0, A^{-1/2}]=[Y_0A^{-1/2}B, A^{-1/2}]=0,\]
	which in turn implies that $Y_0$ is orthogonal to the range of  $D_{(I,I)}\Gamma$ , from \eqref{derivada 1} and \eqref{derivada 2}. So $\Gamma$ is not a submersion at $(I,I)$. 																		
	
\end{proof}

\begin{lem} \label{cuando no conmutan}
	Let $A,B$ in $\defpos$ such that $[A,B]\neq 0$. Then, for some $t_0>0$ sufficiently small, there exists a continuous curve 
	$\gamma(t):(-t_0,t_0) \longrightarrow \matud$ such that $\gamma(0)=I$ and, for $t\neq 0$,
	\[\log \la\left(A^{-1}B\right)\neq \log \la\left(A^{-1}\gamma(t)B\, \gamma(t)^*\right)\prec \log \la\left(A^{-1}B\right).\]
\end{lem}
\begin{proof}
	Let $P$ be an orthogonal projection in $\{A,B\}'$ such that $P$ is minimal, i.e. there is not an orthogonal projection $Q\in \{A,B\}'$ such that $QP=Q$.  Since $AB\neq BA$, $P$ has rank $\geq 2$. Then, by restricting our study to the reducing space $\Pim (P)$, we can assume that $\{A,B\}=\C\cdot I$.
	
	By Lemma \ref{gama es una submersion}, $\Gamma :\matud\times \matud \longrightarrow \defpos ^\tau$ defined as in \eqref{el gama multi}
	is a submersion at $(I,I)$ onto $\defpos ^\tau$. In particular, every continuous curve $l(t):(-1,1)\longrightarrow \defpos ^\tau$ such that $l(0)=A^{-1/2}BA^{-1/2}$ can be lifted (locally) to a continuous curve $\widetilde{l}(t):(-t_0,t_0)\longrightarrow \matud\times \matud$ such that $l(t)=\Gamma\circ \widetilde{l} (t)$, for $-t_0<t<t_0$, such that $\widetilde{l}(0)=(I,I)$.
	
	As a consequence of Theorem \ref{Teo MRS}, $b:=\log \la(A^{-1/2}BA^{-1/2})\neq \left(\log \la(B)-\log \la(A)\right)^\downarrow:=a$.
	Moreover, by Gelfand-Naimark-Lidskii Theorem, $a\prec b$. In particular,
	$\rho(t):=|t|a+(1-|t|)b\prec b$, for $t\in (-1,1)$, with $\rho(t)\neq b$ for $t\neq 0$.
	
	Let $U\in \matud$ be such that $UD_{\la(A^{-1/2}BA^{-1/2})}U^*=A^{-1/2}BA^{-1/2}$, where $D_{\la(A^{-1/2}BA^{-1/2})}$ stands for the diagonal matrix with the eigenvalues of $A^{-1/2}BA^{-1/2}$ arranged in non-increasing order.
	
	Define $l(t):(-1,1)\longrightarrow \defpos ^\tau$ by
	\[l(t)=UD_{\exp (\rho(t))}U^*,\]
	where $\exp(\rho(t)):=(e^{\rho(t)_1},e^{\rho(t)_2},\ldots,e^{\rho(t)_d})\in \R^d$.
	
	Then, $l(t)$ is a continuous curve such that $l(0)=A^{-1/2}BA^{-1/2}$. Moreover, it is clear that $\det (l(t))=\Pi_{i=1}^d e^{\rho(t)_i}=e^{\tr \rho(t)}=\frac{\det(B)}{\det(A)}=\tau$, so 
	$l(t)\in \defpos ^\tau$, $\forall t\in (-1,1)$. Also, by construction, $\log \la(l(t))\prec \log \la (l(0))$, for $t\in (-1,1)$.
	
	Then, for a small $0<t_0$, there is a continuous curve $\widetilde{l}(t)=(U(t),V(t))\in \matud\times \matud$, $t\in (-t_0,t_0)$ such that $\widetilde{l}(0)=(I,I)$ and
	$\Gamma\circ \widetilde{l}(t)=l(t)$. Consider $\gamma(t):(-t_0,t_0)\longrightarrow \matud$ given by $\gamma(t)=U^*(t)V(t)$. Hence,  $\gamma(t)$ is  continuous, $\gamma(0)=I$ and
	\begin{align*}
	\log \la\left (A^{-1}\gamma(t)B\, \gamma^*(t)\right)&=\log \la\left (A^{-1}V^*(t)U(t)BU^*(t)V(t)\right)\\
	&=\log \la\left (V(t)A^{-1/2}U^*(t)V(t)BV^*(t)U(t)A^{-1/2}V^*(t)\right)\\
	&=\log \la (l(t))\prec \log \la (A^{-1}B),
	\end{align*}
	which complete the proof.

\end{proof}

\begin{teo}\label{la curva}
	Let $A,B$ in $\defpos$ such that $$\log \la(A^{-1}B)\neq (\log \la(B)-\log \la(A))^{\downarrow}.$$
	Then, for some $t_0>0$ sufficiently small, there exists a continuous curve 
	$\gamma(t):(-t_0,t_0) \longrightarrow \matud$ such that $\gamma(0)=I$ and
	\[\log \la\left(A^{-1}B\right)\neq \log \la\left(A^{-1}\gamma(t)B\,\gamma(t)^*\right)\prec \log \la\left(A^{-1}B\right) \text{ for } t\neq 0.\]
\end{teo}
\begin{proof}
	By Lemma \ref{cuando no conmutan} it only remains to prove the case $[A,B]=0$.
	Thus, suppose that $AB=BA$. Then, there exists $U\in \matud$ such that
	\[UA^{-1}BU^*=D_{\la(A^{-1})}D_{\la^\sigma(B)},\]
	where $D_{\la(A^{-1})}$ and $D_{\la^\sigma(B)}$ are  diagonal matrices with $\la(A^{-1})=(\la(A)_1^{-1},\la(A)_2^{-1},\ldots, \la(A)_d^{-1})$ arranged in non-decreasing order and $\la^\sigma (B)$ are the eigenvalues of $B$ in some order. Suppose that for some $1\leq i\leq d$, 
	$\la^\sigma_i(B)<\la^\sigma_{i+1}(B)$. Let $\{e_j\}_{j=1}^d$ be an orthonormal basis that $A$ and $B$ has in common associated to $U$.
	
	For simplicity, denote by $\alpha_1=\la(A)_i$, $\alpha_2=\la(A)_{i+1}$ and $\beta_1=\la^\sigma_i(B)$, $\beta_2=\la^\sigma_{i+1}(B)$.
	Then, if we restrict to the eigenspace $\cS$ spanned by the common eigenvalues $\{e_i,e_{i+1}\}$,
	
	\[\restr{UA^{-1}BU^*}{\cS}= \left[ {\begin{array}{cc}
		\frac{\beta_1}{\alpha_1}& 0\\ 
		0 & \frac{\beta_2}{\alpha_2}\\   
		\end{array} } \right] 
	\]
	Let $W(t)\in \matud$, $t\in (-\frac{\pi}{2},\frac{\pi}{2})$ defined by
	\[W(t)=\cos(t)(e_i\otimes e_i+e_{i+1}\otimes e_{i+1})+\sin(t)(e_i\otimes e_{i+1}-e_{i+1}\otimes e_i)+P_{\cS^\perp} .\]
	Then, we consider
	$C_\cS(t):=\restr{UA^{-1}U^*W(t) UBU^*W^*(t)}{\cS}$. Note that $C_\cS(t)$ is such that
	\[C_\cS(t)=\left[ {\begin{array}{cc}
		\alpha_1^{-1} & 0 \\
		0 & \alpha_2^{-1} \\
		\end{array} } \right]\left[ {\begin{array}{cc}
		\cos(t) & \sin(t) \\
		-\sin(t) & \cos(t) \\
		\end{array} } \right] \left[ {\begin{array}{cc}
		\beta_1& 0\\ 
		0 & \beta_2\\  \end{array} } \right] \left[ {\begin{array}{cc}
		\cos(t) & -\sin(t) \\
		\sin(t) & \cos(t) \\
		\end{array} } \right].  
	\]
	It is easy to see that, for $t\neq 0$, 
	$$\tr(C_\cS(t))=\frac{\beta_1}{\alpha_1}+\frac{\beta_2}{\alpha_2}+\sin^2(t)(\beta_1-\beta_2)(\alpha_2^{-1}-\alpha_1^{-1}) <\frac{\beta_1}{\alpha_1}+\frac{\beta_2}{\alpha_2}=\tr(C_\cS(0)).$$ 
	Since $\det(C_\cS(t))=\frac{\beta_1 \beta_2}{\alpha_1 \alpha_2}$ for every $t$, this implies that the eigenvalues $\la_1(C_\cS)\geq \la_2(C_\cS)$ satisfy 
	$$\big(\log(\la_1(C_\cS(t))),\log(\la_2(C_\cS(t)))\big)\prec \left(\log \left(\frac{\beta_1}{\alpha_1}\right),\log \left(\frac{\beta_2}{\alpha_2}\right)\right).$$
	Finally, if $\gamma(t)=U^*W(t)U$, $-\frac{\pi}{2}\leq t\leq \frac{\pi}{2}$, is clear that
	\[\log\la(A^{-1}\gamma(t)B\, \gamma^*(t))=\log(\la(UA^{-1}U^*W(t)UBU^*W^*(t))\prec \log \la (A^{-1}B), \forall t,\]
	since $A^{-1}BP_{\cS^\perp}=A^{-1}\gamma(t)B\, \gamma^*(t)P_{\cS^\perp}$.
\end{proof}
As previously mentioned, a similar analysis of the right inequality in (\ref{Lidskii multiplicativo}) leads to similar conclusions.

\section{An application for the characterization of local extrema in a Procustes problem.}

In what follow, for $B\in\defpos$, consider the unitary orbit
$$\cO_B=\llav{C\in\defpos:\, C=U^*BU \peso{for} U\in\matud}\,.$$
Note that $\cO_B$ is a metric space endowed with the usual metric induced by the operator norm.

Given $A,B\in\defpos$ and $\nui{\cdot}$ an arbitrary unitarily invariant norm in $\mat$,  define the distance 
$F_{\paren{N,\, A,\, B}}=F_N:\cO_B\to \R_{>0}$ given by
\beq\label{deltan}
F_N(C)= N\paren{\log\paren{A^{-1/2}CA^{-1/2}}}\peso{for} C\in\cO_B\,.
\eeq

With this metric, the nearness problem posed in  $\defpos$ is to compute
\beq\label{procuston}
\min_{C\in \cO_B} F_N\paren{C}=\min_{C\in \cO_B}N\paren{\log\paren{A^{-1/2}CA^{-1/2}}}\,.
\eeq
and to characterize the approximants.

In the particular case that $N(\cdot)$ is the Frobenius norm; i.e.  $N(X)^2=\norm{X}_2^2=\tr(X^*X)$ (for $X\in\mat$) 
we have the Riemannian distance

\beq\label{deltados}
F_2(C):=\norm{\log\paren{A^{-1/2} C A^{-1/2}}}_2
=\corch{\sum_{j=1}^n \log^2\paren{\la_j(A^{-1}C)}}^{1/2}
\peso{for} C\in\cO_B\,,
\eeq
so in this case, the Procustes problem  in $\defpos$ is to solve
\beq\label{procusto}
\min_{C\in \cO_B} F_2\paren{C}
=\min_{C\in \cO_B}\corch{\sum_{j=1}^n \log^2\paren{\la_j(A^{-1}C)}}^{1/2}\,.
\eeq

\begin{teo}\label{nuestra funcion}
	Fixed $A,B\in \defpos$ and $N(\cdot)$ a strictly convex u.i.n. in $\mat$ and  consider
	$F_{\paren{N,\, A,\, B}}=F_N:\cO_B\to \R_{>0}$ given by
	$$
	F_N(C)= N\paren{\log\paren{A^{-1/2}CA^{-1/2}}}\,.
	$$
	Then,  if $C_0\in \cO_B$ is a local minimizer of $F_N$, there exists $\cU\in\matud$ such that   
	$U^*AU=D_{\la(A)}$
	and $C_0=U^*BU=D_{\la(B)}$. In particular, $C_0$ is a global minimizer of $F_N$ in $\cO_B$.
	
\end{teo}

\begin{proof}
	Suppose that there is not $U\in\matud$ such that
	$U^*AU=D_{\la(A)}$
	and $C_0=U^*BU=D_{\la(B)}$.
	Then, by Theorem \ref{Teo MRS}, 
	\[\left( \log (\la(C_0)) - \log(\la(A))\right)^\downarrow \neq \log (\la (C_0\,A^{-1}))\,.\] 
	Therefore,  Theorem \ref{la curva} implies that, for some $t_0>0$ sufficiently small, there exists a continuous curve 
	$\gamma(t):(-t_0,t_0) \longrightarrow \matud$ such that $\gamma(0)=I$ and
	\[\log \la\left(A^{-1}C_0\right)\neq 
	\log \la\left(A^{-1/2}\gamma(t)C_0\,\gamma(t)^*A^{-1/2}\right)\prec \log \la\left(A^{-1/2}C_0A^{-1/2}\right) \text{ for } t\neq 0\,.\]
	Then, the fact that $N(\cdot)$ is a strictly convex u.i.n. implies that
	\[N\paren{\log \left(A^{-1/2}\gamma(t)C_0\,\gamma(t)^*A^{-1/2}\right)}< N\paren{\log \left(A^{-1/2}C_0\, A^{-1/2}\right)}\,,\]
	which contradicts the assumption that $C_0$ is a local minimizer of $F_N$ in $\cO_B$.	
	
\end{proof}
\begin{rem}
	In a certain way, Theorem \ref{nuestra funcion} completes the main result of \cite{BhatCon} in the particular case of the family of log-distances. Namely, we  show that the points at the orbit of $B$ described in the statement of the Theorem in \cite{BhatCon} are the only possible global minimizers. In this work we complete that result with  a local study that proves that these global minimizers are also the only local minimizers for the distance function. It is important to mention that almost the same study  can be carried out for the local/global maxima and similar conclusions can be drawn. 
	
	Finally, notice that our conclusions for local extrema can be applied directly to the Procustes problem using the Bures-Wasserstein metric and the Kullback-Leibler divergence as it was studied in  \cite{BhatCon}. Indeed, as the authors in that work show, the analysis of the extrema for the functions involved  also depends on the  Gelfand-Naimark-Lidskii inequality.
\end{rem}

\section*{Aknowledgments}
This work was partially supported by  CONICET 
(PIP 1505/15) and  Universidad Nacional de La Plata (UNLP 11 X829).

\end{document}